\documentclass[11pt]{article}

\usepackage{amsmath,amssymb,amsthm,setspace,graphicx,array}
\usepackage[text={6.5in,8.4in}, top=1.3in]{geometry}

\newtheorem{theorem}{Theorem}[section]

\newtheorem{lemma}[theorem]{Lemma}
\newtheorem{corollary}[theorem]{Corollary}
\newtheorem{conjecture}[theorem]{Conjecture}
\newtheorem{problem}[theorem]{Problem}
\newtheorem{definition}[theorem]{Definition}

\begin{document}
	
\title{
	The number of bounded-degree spanning trees
}
	
\author{
	Raphael Yuster
	\thanks{Department of Mathematics, University of Haifa, Haifa 3498838, Israel. Email: raphael.yuster@gmail.com\,.}
}
	
\date{}
	
\maketitle
	
\setcounter{page}{1}
	
\begin{abstract}
	For a graph $G$, let $c_k(G)$ be the number of spanning trees of $G$ with maximum degree at most $k$.
	For $k \ge 3$, it is proved that every connected $n$-vertex $r$-regular graph $G$ with $r \ge \frac{n}{k+1}$ satisfies
	\vspace*{-2mm}$$
	c_k(G)^{1/n} \ge (1-o_n(1)) r \cdot z_k
	\vspace*{-2mm}$$
	where $z_k > 0$ approaches $1$ extremely fast (e.g. $z_{10}=0.999971$).
	The minimum degree requirement is essentially tight as for every $k \ge 2$ there are connected $n$-vertex $r$-regular graphs $G$ with $r=\lfloor n/(k+1) \rfloor -2$ for which $c_k(G)=0$.
Regularity may be relaxed, replacing $r$ with the geometric mean of the degree sequence
and replacing $z_k$ with $z_k^* > 0$ that also approaches $1$, as long as the maximum degree is at most $n(1-(3+o_k(1))\sqrt{\ln k/k})$. The same holds with no restriction on the maximum degree as long as the minimum degree is at least $\frac{n}{k}(1+o_k(1))$. 

\vspace*{3mm}

\end{abstract}

\section{Introduction} 

For a graph $G$, let $c_k(G)$ be the number of spanning trees of $G$ with maximum degree at most $k$ and let $c(G)$ be the number of spanning trees of $G$. Computationally, these parameters are well-understood: Determining $c(G)$ is easy by the classical Matrix-Tree Theorem
which says that $c(G)$ is equal to any cofactor of the Laplacian matrix of $G$,
while determining $c_k(G)$ is NP-hard for every fixed $k \ge 2$.
In this paper we look at these parameters from the extremal graph-theoretic perspective.
The two extreme cases, i.e. $c(G)$ and $c_2(G)$, are rather well-understood.
As for $c(G)$, Grone and Merris \cite{GM-1988} proved that $c(G) \le (n/(n-1))^{n-1}d(G)/2m$ where
$n$ and $m$ are the number of vertices and edges of $G$ respectively, and $d(G)$ is the product of its degrees.
Note that this upper bound is tight for complete graphs.
Alon \cite{alon-1990}, extending an earlier result of McKay \cite{mckay-1983}, proved that if $G$ is a connected $r$-regular graph, then
$c(G) = (r-o(r))^n$. Alon's method gives meaningful results already for $r=3$, where the proof yields
$(1-o_n(1))c(G)^{1/n} \ge \sqrt{2}$. Alon's result was extended by Kostochka \cite{kostochka-1995} to arbitrary connected graphs with minimum degree $r \ge 3$.
He proved that $c(G) \ge d(G)r^{-nO(\ln r/r)}$ and
improved the aforementioned case of $3$-regular graphs showing that
$(1-o_n(1))c(G)^{1/n} \ge 2^{3/4}$ and that the constant $2^{3/4}$ is optimal.
We mention also that Greenhill, Isaev, Kwan, and McKay \cite{GIKM-2017} asymptotically determined
the expected number of spanning trees in a random graph with a given sparse degree sequence.

The case $c_2(G)$ (the number of Hamilton paths) has a significant body of literature.
All of the following mentioned results hold, in fact, for counting the number of Hamilton {\em cycles}.
First, we recall that there are connected graphs with minimum degree $n/2-1$ for which $c_2(G)=0$,
so most results concerning $c_2(G)$ assume that the graph is {\em Dirac}, i.e. has minimum degree at least 
$n/2$. Dirac's Theorem \cite{dirac-1952} proves that $c_2(G) > 0$ for Dirac graphs.
Significantly strengthening Dirac’s theorem, S\'ark\"ozy, Selkow, and Szemer\'edi \cite{SSS-2003} proved 
that every Dirac graph contains at least $c^n n!$ Hamilton cycles for some small
positive constant $c$. They conjectured that $c$ can be improved to $1/2-o(1)$.
In a breakthrough result, Cuckler and Kahn \cite{CK-2009} settled this conjecture
proving that every Dirac graph with minimum degree $r$ has at least
$(r/e)^n(1-o(1))n$ Hamilton cycles. This bound is tight as shown by an appropriate random graph.
Bounds on the number of Hamilton cycles in Dirac graphs expressed in terms of maximal regular spanning subgraphs were obtained by Ferber, Krivelevich, and Sudakov \cite{FKS-2017}. Their bound matches
the bound of Cuckler and Kahn for graphs that are regular or nearly regular.

In this paper we consider $c_k(G)$ for fixed $k \ge 3$. Observe first that
$c_k(G)^{1/n} \le c(G)^{1/n} < d(G)^{1/n}$ (by simple counting or by the aforementioned result \cite{GM-1988}). Thus, we shall express the lower bounds for $c_k(G)^{1/n}$ in our theorems in terms of
constant multiples of $d(G)^{1/n}$. Notice also that if $G$ is $r$-regular, then $d(G)^{1/n}=r$.

Our first main result concerns connected regular graphs.
It is not difficult to prove that every connected $r$-regular graph with $r \ge n/(k+1)$ has $c_k(G) > 0$
(this also holds for $k=2$ \cite{CS-2013}). We prove that $c_k(G)$ is, in fact, already very large under this minimum degree assumption. To quantify our lower bound we define the following functions of $k$.
$$
f_k = 1-\frac{1}{e}\sum_{i=0}^{k-3}\frac{1}{i!}\;, \qquad g_k = \frac{2}{e(k-1)!}\;.
$$
$$
z_k = 
\begin{cases}
	0.0494, & \text{for } k=3\\
	0.1527, & \text{for } k=4 \\
	(1 - (k+1)(f_k+g_k))^{g_k}(1-g_k)^{1-g_k}{g_k}^{g_k}, & \text{for } k \ge 5\;.
\end{cases}
$$
It is important to observe that $z_k$ approaches $1$ extremely quickly, as Table \ref{table:1}
shows.
\begin{table}
	\vspace{1cm}
	\centering
	\begin{tabular}{|c|c|c|c|c|c|c|c|}
		\hline
		$k$ & $5$ & $6$ & $7$ & $8$ & $9$ & $10$ & $11$\\
		\hline
		$z_k$ & $0.843148$ & $0.962200$ & $0.991935$ & $0.998565$ & $0.999783$ & $0.999971$ & $0.999997$ \\
		\hline
	\end{tabular}
	\caption{The value of $z_k$ for some small $k$.}
	\label{table:1}
\end{table}

\begin{theorem}\label{t:regular}
	Let $k \ge 3$ be given. Every connected $n$-vertex $r$-regular graph $G$ with $r \ge \frac{n}{k+1}$ satisfies
	$$
	c_k(G)^{1/n} \ge (1-o_n(1)) r \cdot z_k\;.
	$$
\end{theorem}
\noindent
The requirement on the minimum degree in Theorem \ref{t:regular} is essentially tight.
In Subsection \ref{sub:construct} we show that for every $k \ge 2$ and for infinitely many $n$,
there are connected $r$-regular graphs $G$ with $r=\lfloor n/(k+1) \rfloor -2$ for which $c_k(G)=0$.
In light of this construction, it may be of some interest to determine whether Theorem \ref{t:regular}
holds with $n/(k+1)-1$ instead of $n/(k+1)$. Furthermore, as our proof of Theorem \ref{t:regular} does
not work for $k=2$, we raise the following interesting problem.
\begin{problem}
	Does there exist a positive constant $z_2$ such that every connected $n$-vertex $r$-regular graph $G$ with $r \ge \frac{n}{3}$ satisfies
	$$
	c_2(G)^{1/n} \ge (1-o_n(1)) r \cdot z_2\;.
	$$
\end{problem}

One may wonder whether the regularity requirement in Theorem \ref{t:regular} can be relaxed,
while still keeping the minimum degree at $n/(k+1)$. It is easy to see that a bound on the {\em maximum} degree cannot be entirely waved. Indeed, consider a complete bipartite graph with one part of order $(n-2)/k$.
It is connected, has minimum degree $(n-2)/k > n/(k+1)$, maximum degree $n-(n-2)/k$ but it clearly does not have any spanning tree with maximum degree at most $k$. However, if we place just a modest restriction on the maximum degree, we can extend Theorem \ref{t:regular}. Let
$$
z^*_k =
\left(1-\frac{1}{7k}\right)^{1-\frac{1}{7k}}\left(\frac{1}{9k}\right)^{\frac{1}{7k}}\;.
$$
It is easy to see that $z^*_k$ approaches $1$. For example, $z^*_{20} > 0.956$.
\begin{theorem}\label{t:nearly-regular}
	There exists a positive integer $k_0$ such that for all $k \ge k_0$ the following holds.
	Every connected $n$-vertex graph $G$ with minimum degree at least $\frac{n}{k+1}$ and maximum degree at most
	$n(1-3\sqrt{\ln k/k})$ satisfies
	$$
	c_k(G)^{1/n} \ge (1-o_n(1)) d(G)^{1/n} \cdot z^*_k\;.
	$$
\end{theorem}

Finally, we obtain a lower bound on $c_k(G)$ where we have no restriction on the maximum degree of $G$. Analogous to Dirac's theorem, Win \cite{win-1975} proved that every connected graph
with minimum degree $(n-1)/k$ has $c_k(G) > 0$ (see also \cite{CFHKT-2001} for an extension of this result). Clearly, the requirement on the minimum degree is tight as the aforementioned example of a complete bipartite graph shows that there are connected graphs with minimum degree $(n-2)/k$
for which $c_k(G)=0$. We prove that for all $k \ge k_0$, if the minimum degree is just slightly larger, then $c_k(G)$ becomes large.

\begin{theorem}\label{t:non-regular}
	There exists a positive integer $k_0$ such that for all $k \ge k_0$ the following holds.
	Every connected $n$-vertex graph $G$ with minimum degree at least $\frac{n}{k}(1+3\sqrt{\ln k/k})$ satisfies
	$$
	c_k(G)^{1/n} \ge (1-o_n(1)) d(G)^{1/n} \cdot z^*_k\;.
	$$
\end{theorem}
Using Szemer\'edi's regularity lemma, it is not too difficult to prove a version of Theorem \ref{t:non-regular} that works already for $k \ge 3$ and where $c_k(G)$ is exponential in $n$. However, the bound we can obtain by that method, after taking its $n$-th root, is not a positive constant multiple of $d(G)^{1/n}$. We do conjecture that the error term in the minimum degree assumption can be eliminated.
\begin{conjecture}
	Let $k \ge 3$. There is a constant $z^\dagger_k > 0$ such that
	every connected $n$-vertex graph $G$ with minimum degree at least $\frac{n}{k}$ satisfies
	$$
	c_k(G)^{1/n} \ge (1-o_n(1)) d(G)^{1/n} \cdot z^\dagger_k
	$$
	where $\lim_{k \rightarrow \infty} z^\dagger_k = 1$.
\end{conjecture}

All of our theorems are based on two major ingredients. The first ingredient consists of proving that $G$ has many spanning forests, each with only a relatively small number of component trees, and each having maximum degree at most $k$. However, the proof of this property varies rather significantly among the various theorems and cases therein.
We combine the probabilistic model of Alon \cite{alon-1990} for showing that there are many {\em out-degree one} orientations with certain properties, together with a novel nibble approach to assemble edges from {\em several} out-degree one orientations.
The second ingredient consists of proving that each of the large spanning forests mentioned above has small ``edit distance'' from a spanning tree with maximum degree at most $k$. Once this is established, it is not difficult to deduce that $G$ has many spanning trees with maximum degree at most $k$.

In Section 2 we prove the edit-distance property. In Section 3 we introduce out-degree one orientations and the {\em multi-stage model} which is the basis for our nibble approach. In Section 4 we consider regular graphs and prove Theorem \ref {t:regular}. In Section 5 we prove Theorems \ref{t:nearly-regular} and
\ref{t:non-regular}.

Throughout the paper we assume that the number of vertices of the host graph, always denoted by $n$, is
sufficiently large as a function of all constants involved. Thus, we refrain from repeatedly mentioning this assumption. We also ignore rounding issues (floors and ceilings) whenever these have no effect on the final statement of our results. We use the terminology {\em $G$-neighbor} of a vertex $v$  to refer to a neighbor of $v$ in $G$, as opposed to a neighbor of $v$ in spanning tree or a spanning forest of $G$.
The notation $d(v)$ always denotes the degree of $v$ in $G$. Other notions that are used are standard,
or defined upon their first use.

\section{Extending a bounded degree forest}\label{sec:extend}

The edit distance between two graphs on the same vertex set is the number of edges in the symmetric difference of their edge sets.
In this section we prove that the edit distance between a bounded degree spanning forest and a bounded degree spanning tree of a graph is proportional to the number of components of the forest, whenever the graph is connected and satisfies a minimum degree condition.

\begin{lemma}\label{l:extend-reg}
	Let $k \ge 3$ and let $G$ be a connected graph with $n$ vertices and minimum degree at least $n/(k+1)$.
	Suppose that $F$ is a spanning forest of $G$ with $m < n-1$ edges and maximum degree at most $k$.
	Furthermore, assume that $F$ has at most $t$ vertices with degree $k$ where
	$t \le n/(6.8k)$.
	Then there exists a spanning forest $F^*$ of $G$ with $m+1$ edges that contains at least $m-3$ edges of $F$.
	Furthermore, $F^*$ has maximum degree at most $k$ and at most $t+4$ vertices with degree $k$.
\end{lemma}
\begin{proof}
	For a forest (or tree) with maximum degree at most $k$, its {\em $W$\!-vertices} are those with degree $k$
	and its {\em $U$\!-vertices} are those with degree less than $k$.
	Denote the tree components of $F$ by $T_1,\ldots,T_{n-m}$.
	Let $U_i \neq \emptyset$ denote the $U$\!-vertices of $T_i$ and let $W_i$ denote the $W$\!-vertices of $T_i$. We distinguish between several cases as follows:\\
	(a) There is some edge of $G$ connecting some $u_i \in U_i$ with some $u_j \in U_j$ where $i \neq j$.\\
	(b) Case (a) does not hold but there is some $T_i$ with fewer than $n/(k+1)$ vertices.\\
	(c) The previous cases do not hold but there is some edge of $G$ connecting some $u_i \in U_i$ to a vertex in a different component of $F$.\\
	(d) The previous cases do not hold.
	
	{\em Case (a).} We can add to $F$ the edge $u_iu_j$ obtaining a forest with $m+1$ edges
	which still has maximum degree at most $k$. The new forest has at most $t+2$ $W$\!-vertices
	since only $u_i$ and $u_j$ increase their degree in the new forest.
	
	\begin{figure}
		\includegraphics[scale=0.8,trim=70 340 305 30, clip]{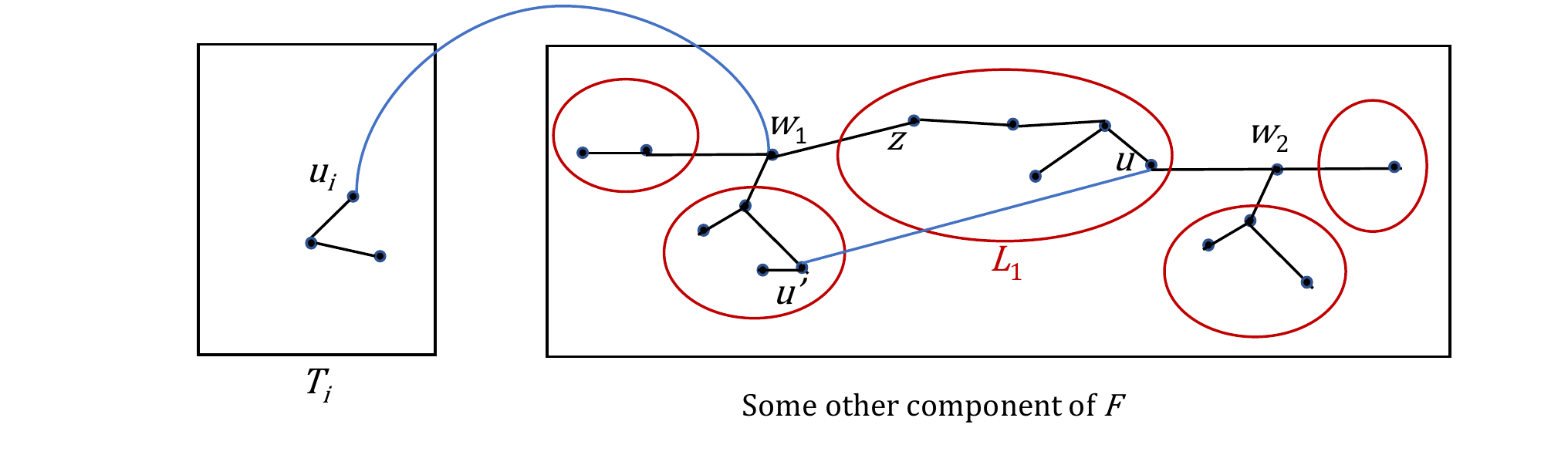}
		\caption{Constructing $F^*$ from $F$ in Case (b) of Lemma \ref{l:extend-reg} (here we use $k=3$). The figure depicts the component $T_i$ containing $u_i$ and some other component containing $w_1$ (and $w_2$ in this example).
			The red ovals depict the various $L_j$'s obtained when removing $w_1$ and $w_2$. The denoted
			$L_1$ contains a vertex $u$ of degree $1$ in $L_1$ (and degree smaller than $3$ in $F$) which has a neighbor $u'$ in $G$ that also has degree smaller than $3$ in $F$. The blue edges represent edges of $G$ that are not used in $F$. To obtain $F^*$ we add $u_iw_1$, add $uu'$ and remove the edge $w_1z$.}
		\label{f:1}
	\end{figure}
	
	{\em Case (b).}
	Let $u_i$ be some vertex with degree $1$ or $0$ in $T_i$ (note that it is possible that $T_i$ is a singleton so that the degree of its unique vertex is indeed $0$ in $T_i$). Since $T_i$ has fewer than $n/(k+1)$ vertices, and since $u_i$ has minimum degree at least $n/(k+1)$
	in $G$, we have that $u_i$ has at least two $G$-neighbors that are not in $T_i$.
	Let $w_1,w_2$ denote such neighbors. Notice that $w_1,w_2$ are $W$\!-vertices of $F$ as we assume Case (a) does not hold.
	
	Assume first that $w_1,w_2$ are adjacent in $F$ (in particular, they are in the same component of $F$). Let $F^*$ be obtained from $F$ by adding both edges $u_iw_1$ and $u_iw_2$ and removing the edge $w_1w_2$. Note that $F^*$ has $m+1$ edges, has $m-1$ edges of $F$, and has maximum degree at most $k$. It also has at most $t+1$ $W$\!-vertices as only $u_i$
	may become a new vertex of degree $k$ (in fact, the degree of $u_i$ in $F^*$ is at most $3$ so if $k > 3$ we still only have $t$ $W$\!-vertices in $F^*$). 
	
	We may now assume that $w_1,w_2$ are independent in $F$.
	Removing both of them from $F$ further introduces at least $2k-1$ component trees
	denoted $L_1,\ldots,L_s$ where $s \ge 2k-1$.
	To see this, observe first that if we remove $w_1$, we obtain at least $k$ nonempty components since $w_1$ has degree $k$. If we then remove $w_2$, we either obtain an additional set of $k$ components
	(if $w_2$ is not in the same component of $w_1$ in $F$) or an additional set of $k-1$ components
	(if $w_2$ is in the same component of $w_1$ in $F$).
	
	Each $L_j$, being a tree, either has at least two vertices of degree $1$, or
	else $L_j$ is a singleton, in which case it has a single vertex with degree $0$ in $L_j$.
	If $L_j$ is a singleton, then its unique vertex has degree at most $2$ in $F$ as it may only be connected in $F$ to $w_1$ and $w_2$. If $L_j$ is not a singleton, then let $v_1,v_2$ be two vertices with degree $1$ in $L_j$. It is impossible for both $v_1,v_2$ to have degree at least $3$ in $F$ as otherwise
	they are both adjacent to $w_1,w_2$ in $F$, implying that $F$ is not a forest (has a $K_{2,2}$).
	In any case, we have shown that each $L_j$ (whether a singleton or not) has a vertex which is
	a $U$\!-vertex of $F$.
	
	Consider now an $L_j$ with smallest cardinality, say $L_1$. Its number of vertices is therefore at most
	\begin{equation}\label{e:k5}
	\frac{n}{s} \le \frac{n}{2k-1}\;.
	\end{equation}
	Let $u$ be a vertex of $L_1$ which is a $U$\!-vertex of $F$.
	By our minimum degree assumption on $G$, $u$ has at least $n/(k+1)-(|V(L_1)|-1)$ neighbors in $G$ that are not in $L_1$. By \eqref{e:k5} we have that
	\begin{equation}\label{e:k5.1}
	\frac{n}{k+1}-(|V(L_1)|-1) \ge \frac{n}{k+1} - \frac{n}{2k-1} >  \frac{n}{6.8k} \ge t\;.
	\end{equation}
	It follows that $u$ has a $G$-neighbor $u'$ not in $L_1$ which is a $U$\!-vertex of $F$.
	Notice that $u$ and $u'$ must be in the same component of $F$ since we assume Case (a) does not hold.
	Since $u'$ is not in $L_1$, adding $uu'$ to $F$ introduces a cycle that contains at least one
	of $w_1,w_2$. Assume wlog that the cycle contains $w_1$ and that $z$ is the neighbor of $w_1$ on the cycle (possibly $z \in \{ u,u' \}$). We can now obtain a forest $F^*$ from $F$ by
	adding $uu'$, adding $u_iw_1$ and removing $w_1z$.
	The obtained forest has $m+1$ edges, has $m-1$ edges of $F$, has maximum degree at most $k$, and at
	most $t+2$ $W$\!-vertices as only $u,u'$ can increase their degree in $F^*$ to $k$.
	Figure \ref{f:1} visualizes $u_i,u,u',w_1,z,L_1$ and the added and removed edges when going from $F$ to $F^*$.
	
	{\em Case (c).}
	In this case, $T_i$ has at least $n/(k+1)$ vertices. Let $w_j \in W_j$ be a $G$-neighbor of $u_i$  in a different component $T_j$ of $F$.
	Removing $w_j$ from $T_j$ splits $T_j \setminus w_j$ into a forest with $k$ component trees $L_1,\ldots,L_k$. So at least one of these
	components, say $L_1$, has at most $(n-|V(T_i)|-1)/k < n/(k+1)$ vertices. Obtain a forest $F^{**}$ from $F$ be adding the edge
	$u_iw_j$ and removing the unique edge of $T_j$ connecting $w_j$ to $L_1$.
	The new forest also has $m$ edges and has $m-1$ edges of $F$. It also has at most $t+1$
	$W$\!-vertices as only $u_i$ may become a new vertex of degree $k$.
	But in $F^{**}$, there is a component, namely $L_1$, with fewer than $n/(k+1)$ vertices.
	Hence, we arrive at either Case (a) or Case (b) for $F^{**}$.
	So, applying the proofs of these cases to $F^{**}$ (and observing that the number of $W$\!-vertices in $F^{**}$ is only $t+1$ so \eqref{e:k5.1} still holds because of the slack in the sharp inequality of \eqref{e:k5.1}),
	we obtain a forest $F^*$ with $m+1$ edges, at least $m-2$ edges of $F$, maximum degree at most $k$,
	and at most $t+3$ $W$\!-vertices.
	
	{\em Case (d).}
	Since $G$ is connected, we still have an edge of $G$ connecting some vertex $w_i \in W_i$ with some $w_j \in W_j$. Without loss of generality, $|V(T_j)|\le n/2$. 
	Removing $w_j$ from $T_j$ splits $T_j \setminus w_j$ into a forest with $k$ component trees $L_1,\ldots,L_k$. So at least one of these components, say $L_1$, has at most
	$|V(T_j)|/k \le n/(2k)$ vertices.
	Let $u$ be a vertex of $L_1$ of degree $1$ in $F$. So, $u$ has at least $n/(k+1)-n/(2k) > n/(6.8k) \ge t$
	neighbors not in $L_1$.
	It follows that $u$ has a $G$-neighbor $u'$ which is a $U$\!-vertex of $F$.
	Also notice that $u' \in T_j$ since we assume Case (a) does not hold.
	Now, let $F^{**}$ be obtained from $F$ by adding the edge $uu'$ and removing
	the unique edge of $T_j$ connecting $w_j$ to $L_1$.
	The new forest also has $m$ edges and has $m-1$ edges of $F$. It also has at most $t+1$
	$W$\!-vertices as only $u'$ may become a new vertex of degree $k$.
	But observe that in $F^{**}$ the degree of $w_j$ is only $k-1$. Since $w_j$ has a $G$-neighbor (namely $w_i$) in a different component of $F^{**}$, we arrive in $F^{**}$  
	at either Case (a) or Case (b) or Case (c).
	So, applying the proofs of these cases to $F^{**}$ (and observing that the number of $W$\!-vertices in $F^{**}$ is only $t+1$ so \eqref{e:k5.1} still holds because of the slack in the sharp inequality of \eqref{e:k5.1}),
	we obtain a forest $F^*$ with $m+1$ edges, at least $m-3$ edges of $F$, maximum degree at most $k$,
	and at most $t+4$ $W$\!-vertices.
\end{proof}

By repeated applications of Lemma \ref{l:extend-reg} where we start with a large forest and repeatedly increase the number of edges until obtaining a spanning tree, we immediately obtain the following corollary.
\begin{corollary}\label{coro:extend-reg}
	Let $k \ge 3$ and let $G$ be a connected graph with $n$ vertices and minimum degree at least $n/(k+1)$.
	Suppose that $F$ is a spanning forest of $G$ with $n-O(\ln n)$ edges and maximum degree at most $k$.
	Furthermore, assume that $F$ has at most $t$ vertices with degree $k$ where
	$t \le n/(7k)$.
	Then there exists a spanning tree of $G$ with maximum degree at most $k$ where all but at most $O(\ln n)$ of its edges are from $F$. \qed
\end{corollary}
	
\section{From out-degree one orientations to bounded degree spanning trees}\label{sec:model}

Let $G$ be a graph with no isolated vertices. An {\em out-degree one orientation} of $G$ is obtained by letting each vertex $v$ of $G$
choose precisely one of its neighbors, say $u$, and orient the edge $vu$ as $(v,u)$ (i.e from $v$ to $u$).
Observe that an out-degree one orientation may have cycles of length $2$.
Also note that an out-degree one orientation has the property that each component\footnote{A component of a directed graph is a component of its underlying undirected graph.} 
contains precisely one directed cycle and that all cycles in the underlying graph of an out-degree one orientation are directed cycles. Furthermore, observe that the edges of the component that are not on its unique directed cycle (if there are any) are oriented ``toward'' the cycle. In particular, given the cycle, the orientation of each non-cycle edge of the component is uniquely determined.
Let ${\cal H}(G)$ denote the set of all out-degree one orientations of $G$. Clearly, $|{\cal H}(G)|=d(G)$.

Most of our proofs use the probabilistic model of Alon \cite{alon-1990}: Each $v \in V(G)$ chooses independently and uniformly at random a neighbor $u$ and the edge $vu$ is oriented $(v,u)$. In this way we obtain a uniform probability distribution over the sample space ${\cal H}(G)$. We let $\vec{G}$ denote a randomly selected element of ${\cal H}(G)$ and let $\Gamma(v)$ denote the chosen out-neighbor of $v$.

We focus on certain parameterized subsets of ${\cal H}(G)$. Let ${\cal H}_{k,s}(G)$ be the subset of all elements of ${\cal H}(G)$ with maximum
in-degree at most $k-1$ and with at most $s$ vertices of in-degree $k-1$. If $s = n$ (i.e. we do not
restrict the number of vertices with in-degree $k-1$) then we simply denote the set by ${\cal H}_k(G)$.
Let ${\cal H}^*_\ell(G)$ be the subset of all elements of ${\cal H}(G)$ with at most
$\ell$ directed cycles (equivalently, at most $\ell$ components). Our proofs are mostly concerned with establishing lower bounds for the probability that 
$\vec{G} \in {\cal H}_{k,s}(G) \cap {\cal H}^*_\ell(G)$. Hence we denote
$$
P_{k,s,\ell}(G) = \Pr[\vec{G} \in {\cal H}_{k,s}(G) \cap {\cal H}^*_\ell(G)]\;.
$$

\begin{lemma}\label{l:completing}
	Let $k \ge 3$ be given. Suppose that $G$ is a connected graph with minimum degree at least $n/(k+1)$.
	Then:
	$$
	c_k(G)^{1/n} \ge (1-o_n(1))d(G)^{1/n} P_{k,n/(7k),\ln n}(G)^{1/n}\;.
	$$
\end{lemma}
\begin{proof}
	Let $p=P_{k,n/(7k),\ln n}(G)$.
	By the definition of $p$, we have that
	$$
	|{\cal H}_{k,n/(7k)}(G) \cap {\cal H}^*_{\ln n}(G)| \ge d(G)p\;.
	$$
	Consider some $\vec{G} \in {\cal H}_{k,n/(7k)}(G) \cap {\cal H}^*_{\ln n}(G)$.
	As it has at most $\ln n$ directed cycles (and recall that these cycles are pairwise vertex-disjoint
	as each belongs to a distinct component), it has at most  $\ln n$ edges that, once removed
	from $\vec{G}$, turn it into a forest $F$ with at least $n-O(\ln n)$ edges.
	Viewed as an undirected graph, $F$ has maximum degree at most $k$ (since the in-degree of each vertex of $\vec{G}$ is at most $k-1$ and the out-degree of each vertex of $\vec{G}$ is precisely $1$).
	Thus, we have a mapping assigning each
	$\vec{G} \in {\cal H}_{k,n/(7k)}(G) \cap {\cal H}^*_{\ln n}(G)$ an undirected forest $F$.	
	While this mapping is not injective, the fact that $\vec{G}$ only has at most $\ln n$ components
	implies that each $F$ is the image of at most
	$n^{O(\ln n)}$ distinct $\vec{G}$. Indeed, given an undirected $F$, suppose it has $t \le \ln n$ components of sizes $s_1,\ldots,s_t$. To turn it into an element of ${\cal H}(G)$, we must first add a single edge to each component to give a cycle, and then choose the orientation of each cycle in each
	component, which implies the orientation of non-cycle edges. Hence, the number of possible
	$\vec{G}$ obtained from $F$ is at most $\prod_{i=1}^t (2s_i^2) \le n^{O(\ln n)}$.
	Furthermore, since $\vec{G}$ has at most $n/(7k)$ vertices with in-degree $k-1$, it follows that
	$F$ has at most $n/(7k)$ vertices with degree $k$. By Corollary \ref{coro:extend-reg},
	there exists a spanning tree $T$ of $G$ with maximum degree at most $k$ where all but at most
	$O(\ln n)$ of its edges are from $F$. Thus, we have a mapping assigning each
	$\vec{G} \in {\cal H}_{k,n/(7k)}(G) \cap {\cal H}^*_{\ln n}(G)$ a spanning tree $T$ of $G$ with maximum degree at most $k$. While this mapping is not injective, the fact that the edit distance
	between $T$ and $F$ is $O(\ln n)$ trivially implies that each $T$ is the image of at most
	$n^{O(\ln n)}$ distinct $F$. Hence, we obtain that
	$$
	c_k(G) \ge d(G)p n^{-O(\ln n)}\;.
	$$
	Taking the $n$'th root from both sides of the last inequality therefore concludes the lemma.
\end{proof}

We also require an upper bound for the probability that  $\vec{G}$ has many components.
The following lemma is proved by Kostochka \cite{kostochka-1995} (see Lemma 2 in that paper, applied to the case where the minimum degree is at least $n/(k+1)$, as we assume).
\begin{lemma}\label{l:compoents}\cite{kostochka-1995}
	Let $G$ be a graph with minimum degree at least $n/(k+1)$.
	The expected number of components of $\vec{G}$ is at most $(k+1)\ln n$. \qed
\end{lemma}

For $\vec{G} \in {\cal H}(G)$, let $B^{\vec{G}}_i$ denote the set of vertices with in-degree $i$.
We will omit the superscript and simply write $B_i$ whenever  $\vec{G}$ is clear from context.

We define the following {\em $K$-stage model} for establishing a random element of ${\cal H}(G)$.
This model is associated with a positive integer $K$ and a convex sum of probabilities $p_1+\cdots+p_K=1$.
In the first part of the $K$-stage model, we select uniformly and independently (with replacement) $K$ elements of ${\cal H}(G)$ as in the aforementioned model of Alon. Denote the selected elements by $\vec{G}_c$ for $c=1,\ldots,K$.
Let $\Gamma_c(v)$ denote the out-neighbor of $v$ in $\vec{G}_c$.
In the second part of the $K$-stage model, we let each vertex $v \in V(G)$ choose precisely one of
$\Gamma_1(v),\ldots,\Gamma_K(v)$ where $\Gamma_c(v)$ is chosen with probability $p_c$.
Observe that the resulting final element $\vec{G}$ consisting of all $n=|V(G)|$ selected edges is also a uniform random element of ${\cal H}(G)$. Also note that for any given 
partition of $V(G)$ into parts $V_1,\ldots,V_K$, the probability that all out-edges of the vertices of $V_c$ are taken from $\vec{G}_c$ for all $c=1,\ldots,K$ is precisely $\prod p_c^{|V_c|}$.

As mentioned in the introduction, most of our proofs for lower-bounding $c_k(G)$ contain two major ingredients.
The first ingredient consists of using the $K$-stage model for a suitable $K$ in order to
establish a lower bound for $P_{k,s,\ell}(G)$ (with $\ell = \ln n$). This first ingredient further splits into several steps:\\
a) The {\em nibble step} where we prove that with nonnegligible probability, there is a forest with a linear number of edges consisting of edges of $\vec{G}_1,\ldots,\vec{G}_{K-1}$ and which satisfies certain desirable properties.\\
b)  The {\em completion step} where we prove that given a forest with the properties of the nibble step we can, with nonnegligible probability, complete it into an out-degree one orientation with certain desirable properties using only the edges
of $\vec{G}_{K}$.\\
c) A {\em combination lemma} which uses (a) and (b) above to prove a lower bound for $P_{k,s,\ell}(G)$.\\
The second ingredient uses Lemma \ref{l:completing} applied to
the lower bound obtained in (c) to yield the final outcome of the desired proof.
Table \ref{table:2} gives a roadmap for the various lemmas used for establishing steps steps (a) (b) (c), and the value of $K$ used.
\begin{table}
	\vspace{1cm}
	\centering
	\setlength\extrarowheight{4pt}
	\begin{tabular}{l|c|c|c|c}
		\hline
		Theorem or case thereof & $K$ & nibble step & completion step & combination \\
		\hline
		\ref{t:regular}, $k \ge 5$ & $2$ & Lemma \ref{l:red-forest} & Lemma \ref{l:successful-not3} & Lemma \ref{l:regular}  \\
		\hline
		\ref{t:regular}, $k=4$ & $5$ & Lemma \ref{l:a-f} & Lemma \ref{l:successful-3} & Lemma \ref{l:regular-3}  \\
		\hline
		\ref{t:regular}, $k=3$ & $20$ & Lemma \ref{l:a-f} & Lemma \ref{l:successful-3} & Lemma \ref{l:regular-3}  \\
		\hline
		\ref{t:nearly-regular}  & $2$ & Lemma \ref{l:red-forest-2} & Lemma \ref{l:successful-2} & Lemma \ref{l:nearly-regular}  \\
		\hline
	\end{tabular}
	\caption{A roadmap for the proofs of Theorems \ref{t:regular}, \ref{t:nearly-regular}, \ref{t:non-regular}.}
	\label{table:2}
\end{table}

\section{Proof of Theorem \ref{t:regular}}\label{sec:regular}

In this section we assume that $G$ is $r$-regular with $r \ge n/(k+1)$.
We will consistently be referring to the notation of Section \ref{sec:model}.
When $k \ge 5$ we will use the two-stage model ($K=2$) and when $k \in \{3,4\}$ (dealt with in the next subsection)
we will need to use larger $K$ (see Table \ref{table:2}).

\subsection{The case $k \ge 5$}
We first need to establish several lemmas (the first lemma being straightforward).
\begin{lemma}\label{l:prob-indeg}
	Let $G$ be an $r$-regular graph with $r \ge n/(k+1)$. For $0 \le i \le n$, the probability that
	$v \in B_i$ (i.e., that $v$ has in-degree $i$ in $\vec{G}$) is
	$$
	\Pr[v \in B_i ] =  \binom{r}{i}\frac{1}{r^i}\left(1-\frac{1}{r}\right)^{r-i} \le (1+o_n(1))\frac{1}{i!e}\;.
	$$
	Furthermore, the in-degree of $v$ in  $\vec{G}$ is nearly Poisson as for all $0 \le i \le k$,
	\[
	\pushQED{\qed} 
	\Pr[v \in B_i] = \left(1 \pm O(n^{-1})\right)\frac{1}{i!e}\;. \qedhere 
	\popQED
	\]
\end{lemma}
\noindent
\begin{lemma}\label{l:neighbors-ok}
	Let $G$ be an $r$-regular graph with $r \ge n/(k+1)$. For all $0 \le i \le k$ and for any set $X$ of vertices of $G$ it holds that
	$$
	\Pr \left[\, \left| |X \cap B_i| - \frac{|X|}{i!e} \right| > n^{2/3}\right] < \frac{1}{n^2}\;.
	$$
\end{lemma}
\begin{proof}
	Consider the random variable $|X \cap B_i|$. By Lemma \ref{l:prob-indeg}, its expectation, denoted by
	$X_0$, is $X_0 = (1 \pm O(\frac{1}{n}))\frac{|X|}{i!e}=\frac{|X|}{i!e} \pm O_n(1)$. Now, suppose we expose the edges of $\vec{G}$ one by one
	in $n$ steps (in each step we choose the out-neighbor of another vertex of $G$), and let $X_j$ be the expectation of $|X \cap B_i|$ after $j$ steps have been exposed (so after the final stage we have $X_n=|X \cap B_i|$). Then $X_0,X_1,\ldots,X_n$ is a martingale satisfying the Lipschitz condition (each exposure increases by one the in-degree of a single vertex), so by Azuma's inequality
	(see \cite{AS-2004}), for all $\lambda > 0$,
	$$
	\Pr \left[\, \left| |X \cap B_i| - X_0 \right| > \lambda \sqrt{n} \right] < 2e^{-\lambda^2/2}\;.
	$$
	Using, say, $\lambda = n^{1/10}$ the lemma immediately follows.
\end{proof}

\begin{lemma}\label{l:removal}
	Let $G$ be an $r$-regular graph with $r \ge n/(k+1)$.
	For all $3 \le t \le k$ the following holds: With probability at least $\frac{1}{10}$, $\vec{G}$ has a set of at most
	$\frac{5.9n}{3e t!}$ edges, such that after their removal, the remaining subgraph has maximum in-degree at most
	$t-1$.
\end{lemma}
\begin{proof}
Let 
$$
Q_{\vec{G},t} = \sum_{i=t}^n (i-t+1)|B_i|
$$
be the smallest number of edges we may delete from $\vec{G}$ in order to obtain a subgraph where all vertices have in-degree at most $t-1$. We upper-bound the expected value of $Q_{\vec{G},t}$. By Lemma \ref{l:prob-indeg} we have that
$$
\mathbb{E}[Q_{\vec{G},t}] = \sum_{i=t}^n (i-t+1)\mathbb{E}[|B_i|] \le (1+o_n(1))\frac{n}{e}\left(\sum_{i=t}^n (i-t+1)\frac{1}{i!}\right)\;.
$$
Now, for all $t \ge 4$, each term in the sum $\sum_{i=t}^n (i-t+1)\frac{1}{i!}$ is smaller than its predecessor
by at least a factor of $2.5$, which means that for all $n$ sufficiently large
$$
\mathbb{E}[Q_{\vec{G},t}] \le \frac{5.3n}{3e t!}\;.
$$
It is easily verified that for $t=3$, the last inequality also holds since
$\sum_{i=3}^\infty \frac{i-2}{i!} < 5.3/18$.
By Markov's inequality, we therefore have that with probability at least $\frac{1}{10}$, for $t \ge 3$ it holds that
$$
Q_{\vec{G},t} \le \frac{5.9n}{3e t!}\;. 
$$
Thus, with probability at least $\frac{1}{10}$, we can pick a set of at most
$\frac{5.9n}{3e t!}$ edges of $\vec{G}$, such that after their removal, the remaining subgraph has maximum in-degree at most $t-1$. 
\end{proof}

\begin{lemma}\label{l:red-forest}
	Let $G$ be an $r$-regular graph with $r \ge n/(k+1)$.
	With probability at least $\frac{1}{20}$, $\vec{G}$ has a spanning forest $F$ such that:\\
	(a) $F$ has maximum in-degree at most $k-2$.\\
	(b) $F$ has at least $n-\frac{2n}{e (k-1)!}$ edges.\\
	(c) The number of vertices of $F$ with in-degree at most $k-3$ is at least $(1-o_n(1))\frac{n}{e}\sum_{i=0}^{k-3}\frac{1}{i!}$.
\end{lemma}
\begin{proof}
By Lemma \ref{l:removal}, with probability at least $\frac{1}{10}$ we can remove at most
	$\frac{5.9n}{3e (k-1)!}$ edges from $\vec{G}$, such that after their removal, the remaining subgraph has maximum in-degree at most $k-2$.
	
	By Lemma \ref{l:compoents}, with probability at most  $\frac{1}{40}$ we have that $\vec{G}$ has more than $40(k+1)\ln n$ components. Recalling that in $\vec{G}$ each component can be made
	a tree by removing a single edge from its unique directed cycle, with probability at least  $1-\frac{1}{40}$ we have that $\vec{G}$ can be made acyclic by removing at most  $40(k+1)\ln n$
	edges.
	
	By Lemma \ref{l:neighbors-ok} applied to $X=V(G)$,  with probability at least
	$1-(k-1)/n^2 > 1-1/40$
	we have that for all $0 \le i \le k-3$, the number of vertices of $\vec{G}$ with in-degree
	$i$ is at least $n/(i!e)-n^{2/3} \ge (1-o_n(1))n/(i!e)$. Thus, with probability at least
	$1-1/40$ there are at least $(1-o_n(1))\frac{n}{e}\sum_{i=0}^{k-3}\frac{1}{i!}$ vertices of $\vec{G}$
	with in-degree at most $k-3$.
	
	We therefore obtain that with probability at least
	$\frac{1}{10}-\frac{1}{40}-\frac{1}{40}=\frac{1}{20}$, the claimed forest exists and has at least $n-\frac{5.9n}{3e (k-1)!}-40(k+1)\ln n \ge
	n-\frac{2n}{e (k-1)!}$ edges.
\end{proof}

Using the two-stage model, consider $\vec{G}_1$ and $\vec{G}_2$ as denoted in Section \ref{sec:model}. We say that $\vec{G}_1$ is {\em successful}
if it has a spanning forest as guaranteed by Lemma \ref{l:red-forest}. By that lemma, with probability at least $\frac{1}{20}$, we have that $\vec{G}_1$ is successful.
Assuming $\vec{G}_1$ is successful, designate a spanning forest $F_1$
of it satisfying the properties of Lemma \ref{l:red-forest}. Let $X_1 \subset V(G)$ be the set of vertices with
out-degree $0$ in $F_1$. Thus, we have by Lemma \ref{l:red-forest} that $|X_1| \le \frac{2n}{e(k-1)!}=ng_k$.

Now, consider the set of edges of $\vec{G}_2$ emanating from $X_1$, denoting them by
$E_2 = \{(v,\Gamma_2(v))\,|\, v \in X_1\}$. By adding $E_2$ to $F_1$ we therefore obtain an out-degree one orientation of $G$, which we denote (slightly abusing notation) by $E_2 \cup F_1$.
\begin{lemma}\label{l:successful-not3}
	Suppose that $k \ge 5$. Given that $\vec{G}_1$ is successful, and given the corresponding forest $F_1$,
	the probability that $(E_2 \cup F_1) \in {\cal H}_{k-1}(G) \cap {\cal H}^*_{\ln n}(G)$ is at least
$$
(1 - (k+1)(f_k+g_k)-o_n(1))^{n g_k}\;.
$$
\end{lemma}
\begin{proof}
	Fix an arbitrary ordering of the vertices of $X_1$, say $v_1,\ldots,v_{|X_1|}$. We consider the edges $(v_i,\Gamma_2(v_i))$ one by one, and let $E_{2,i} \cup F_1$ be the graph obtained after adding to
	$F_1$ the edges $(v_j,\Gamma_2(v_j))$ for $1 \le j \le i$. Also let $E_{2,0} = \emptyset$.
	We say that $E_{2,i} \cup F_1$ is {\em good} if it satisfies the following two properties:\\
	(i) The in-degree of each vertex in $E_{2,i} \cup F_1$ is at most $k-2$.\\
	(ii) Every component of $E_{2,i} \cup F_1$ with fewer than $n/\ln n$ vertices is a tree.
	
	Trivially, $E_{2,0} \cup F_1=F_1$ is good, since $F_1$ is a forest where the in-degree of each vertex is
	at most $k-2$. We estimate the probability that $E_{2,i+1} \cup F_1$ is good given that $E_{2,i} \cup F_1$ is good.
	
	Consider vertex $v_{i+1}$. By Property (c) of Lemma \ref{l:red-forest},
	$v_{i+1}$ has at most $(1+o_n(1))nf_k$ neighbors with in-degree $k-2$ in $F_1$ (recall that
	$f_k = 1-\frac{1}{e}\sum_{i=0}^{k-3}\frac{1}{i!}$).
	Thus, there is a subset $S$ of at least $r-f_k(1+o_n(1))n-i$ neighbors of $v_{i+1}$ in $G$ which still have in-degree at most $k-3$ in $E_{2,i} \cup F_1$.
	Now, if the component of $v_{i+1}$ in $E_{2,i} \cup F_1$ has fewer than $n/\ln n$ vertices, then further remove from $S$ all vertices of that component. In any case, $|S| \ge r-f_k(1+o_n(1))n-i- n/\ln n$. The probability that $\Gamma_2(v_{i+1}) \in S$ is therefore at least
	\begin{align*}
	\frac{r-f_k(1+o_n(1))n-i-\frac{n}{\ln n}}{r} & = 1 - \frac{f_k(1+o_n(1))n+i+\frac{n}{\ln n}}{r}\\
	& \ge 1 - \frac{f_k(1+o_n(1))n+\frac{2n}{e(k-1)!}+\frac{n}{\ln n}}{n/(k+1)}\\
	& = 1 - (k+1)(f_k+g_k)-o_n(1)\;.
	\end{align*}
Now, to have $\Gamma_2(v_{i+1}) \in S$ means that we are not creating any new components of size smaller than $n/\ln n$, so all components of size at most $n/\ln n$ up until now are still trees. It further means that $E_{2,i+1} \cup F_1$ still has maximum in-degree at most
	$k-2$. In other words, it means that $E_{2,i+1} \cup F_1$ is good.
	We have therefore proved that the final $E_2 \cup F_1$ is good with probability at least
	$$
	(1 - (k+1)(f_k+g_k)-o_n(1))^{|X_1|} \ge (1 - (k+1)(f_k+g_k)-o_n(1))^{n g_k}\;.
	$$
	Finally, observe that for $E_2 \cup F_1$ to be good simply means that it belongs to ${\cal H}_{k-1}(G) \cap {\cal H}^*_{\ln n}(G)$.
\end{proof}

\begin{lemma}\label{l:regular}
	Let $k \ge 5$. Then,
	$$
	P_{k,0,\ln n}(G)^{1/n} \ge (1-o_n(1))z_k\;.
	$$
\end{lemma}
\begin{proof}
	Using the two-stage model, we have by Lemma \ref{l:red-forest} that $\vec{G}_1$ is successful with probability at least $\frac{1}{20}$. Thus, by Lemma \ref{l:successful-not3}, with probability at least
	$$
	\frac{1}{20}(1 - (k+1)(f_k+g_k)-o_n(1))^{n g_k}
	$$
	the following holds:
	There is an out-degree one orientation $\vec{G}$ consisting of $x \ge n-\frac{2n}{e(k-1)!}$ edges
	of $\vec{G}_1$, and hence at most $n-x \le ng_k$ edges of $\vec{G}_2$, which is in ${\cal H}_{k-1}(G) \cap {\cal H}^*_{\ln n}(G)$ (observe that being in ${\cal H}_{k-1}(G)$ is the same as being in ${\cal H}_{k,0}$, i.e. there are zero vertices with in-degree $k-1$ since every vertex has maximum in-degree at most $k-2$).
	
	Assuming that this holds, let $X$ be the set of vertices whose out-edge in $\vec{G}$ is from $\vec{G}_1$. Now let $p_1+p_2=1$ be the probabilities associated with the two-stage model where we will
	use $p_2 < \frac{1}{2}$. The probability that in the second part of the two-stage model,
	each vertex $v \in X$ will indeed choose $\Gamma_1(v)$ and each vertex $v \in V(G) \setminus X$ will indeed choose $\Gamma_2(v)$ is precisely
	$$
	p_1^{x}{p_2}^{n-x} \ge (1-p_2)^{n-ng_k}{p_2}^{ng_k}\;.
	$$
	Optimizing, we will choose $p_2=g_k$.
	Recalling that the final outcome of the two-stage model is a completely random element of ${\cal H}(G)$,
	we have that
	$$
	P_{k,0,\ln n}(G) \ge
	\frac{1}{20}(1 - (k+1)(f_k+g_k)-o_n(1))^{n g_k}
	(1-g_k)^{n-ng_k}{g_k}^{ng_k}\;.
	$$
	Taking the $n$'th root from both sides and recalling that
	$z_k=(1 - (k+1)(f_k+g_k))^{g_k}(1-g_k)^{1-g_k}{g_k}^{g_k}$ yields the lemma.
\end{proof}

\begin{proof}[Proof of Theorem \ref{t:regular} for $k \ge 5$]
	By Lemma \ref{l:regular} we have that $P_{k,0,\ln n}(G)^{1/n} \ge (1-o_n(1))z_k$.
	As trivially $P_{k,0,\ln n}(G) \le P_{k,n/(7k),\ln n}(G)$ we have by Lemma \ref{l:completing}
	that
	$$
	c_k(G)^{1/n} \ge (1-o_n(1))d(G)^{1/n} (1-o_n(1))z_k = (1-o_n(1))r \cdot z_k\;.
	$$
\end{proof}

\subsection{The cases $k =3$ and $k=4$}

Lemma \ref{l:successful-not3} doesn't quite work when $k \in \{3,4\}$ as the constant
$1-(k+1)(f_k+g_k)$ is negative in this case ($f_4=1-2/e$ and $g_4=1/(3e)$).
To overcome this, we need to make several considerable adjustments in our arguments.
Among others, this will require using the $K$-stage model for $K$ relatively large
($K = 20$ when $k=3$ and $K=5$ when $k=4$ will suffice).
Recall that in this model we have randomly chosen out-degree one orientations $\vec{G}_1,\ldots,\vec{G}_K$.
Define the following sequence:
$$
q_i = 
\begin{cases}
	\frac{1}{e}, & \text{for } i=1\\
	\frac{q_{i-1}}{e^{q_{i-1}}}, & \text{for } i > 1\;.
\end{cases}
$$
Slightly abusing notation, for sets of edges $F_1,\ldots,F_i$ where $F_j \subset E(\vec{G}_j)$ we let $\cup_{j=1}^i F_j$ denote the graph whose edge set is the union of these edge sets.
\begin{definition}\label{def:successful}
	For $1 \le i \le K-1$ we say that $\vec{G}_i$ is {\em successful} if $\vec{G}_i$ has a subset of edges $F_i$ such that all the following hold:\\
	(a) $i=1$ or $\vec{G}_{i-1}$ is successful (so the definition is recursive).\\
	(b) $F_1,\ldots,F_i$ are pairwise-disjoint and $\cup_{j=1}^i F_j$ is a forest.\\
	(c) The maximum in-degree and maximum out-degree of $\cup_{j=1}^i F_j$ is at most $1$.\\
	(d) $\cup_{j=1}^i F_j$ has $(1 \pm o_n(1))nq_i$ vertices with in-degree $0$.\\
	(e) For all $v \in V(G)$ the number of $G$-neighbors of $v$ having in-degree $0$ in
	$\cup_{j=1}^i F_j$ is $(1 \pm o_n(1))rq_i$.\\
	(f) For all $v \in V(G)$ the number of $G$-neighbors of $v$ having out-degree $0$ in
	$\cup_{j=1}^i F_j$ is $(1 \pm o_n(1))rq_i$.
\end{definition}
\begin{lemma}\label{l:a-f}
	For all $1 \le i \le K-1$, $\vec{G}_i$ is successful with probability at least $\frac{1}{2^i}$.
\end{lemma}
\begin{proof}
	We prove the lemma by induction. Observe that for $i \ge 2$, it suffices to prove that, given that $\vec{G}_{i-1}$ is successful, then $\vec{G}_i$ is also successful with probability $\frac{1}{2}$. For the base case $i=1$ it just suffices to prove
	that items (b)-(f) in Definition \ref{def:successful} hold with probability $\frac{1}{2}$ without the preconditioning item (a), so this is easier than
	proving the induction step; thus we shall only prove the induction step.
	In other words, we assume that $\vec{G}_{i-1}$ is successful and given this assumption, we prove that $\vec{G}_i$ is successful with probability $\frac{1}{2}$.
	
	For notational convenience, let $F = \cup_{j=1}^{i-1} F_j$.
	Let $X_{i-1}$ be the set of vertices with out-degree $0$ in $F$. Since $\vec{G}_{i-1}$ is successful
	we have that $|X_{i-1}|=(1 \pm o_n(1))nq_{i-1}$ (in a digraph with maximum in-degree $1$ and maximum out-degree $1$, the number of vertices with in-degree $0$ equals the number of vertices with out-degree $0$). Consider the set of edges of $\vec{G}_i$ emanating from $X_{i-1}$, denoting them by $E_i = \{(v,\Gamma_i(v))\,|\, v \in X_{i-1}\}$. By adding $E_i$ to $F$ we therefore obtain an out-degree one orientation of $G$, which we denote by $E_i \cup F$. We would like to prove that by deleting just a small amount of edges from $E_i$, we have a subset $F_i \subset E_i$ such that $F_i \cup F$ satisfies items (b)-(f) of Definition \ref{def:successful}. Fix some ordering of $X_{i-1}$, say $v_1,\ldots,v_{|X_{i-1}|}$.
	Let $E_{i,h} \cup F$ be the graph obtained after adding to
	$F$ the edges $(v_j,\Gamma_i(v_j))$ for $1 \le j \le h$. Also let $E_{i,0} = \emptyset$.
	
	We start by taking care of Item (b). For $0 \le h < |X_{i-1}|$, we call $v_{h+1}$ {\em friendly} if the component of $v_{h+1}$ in
	$E_{i,h} \cup F$ has at most $\sqrt{n}$ vertices and $\Gamma_i(v_{h+1})$ belongs to that component. The probability of being friendly is therefore at most $\sqrt{n}/r$, so the expected number of friendly vertices is at most
	$|X_{i-1}|\sqrt{n}/r \le (1\pm o_n(1))nq_{i-1}\sqrt{n}/(n/5) < n^{2/3}$
	(recall that we assume that $r \ge n/(k+1)$ and that $k \in \{3,4\}$ so $r \ge n/5$).
	By Markov's inequality, with probability $p_{(b)}=1-o_n(1)$, there are at  most $n^{3/4}$ friendly
	vertices. But observe that removing from $E_i \cup F$ the edges of $E_i$ emanating from friendly vertices results in a digraph with maximum out-degree $1$ in which every component with at most $\sqrt{n}$ vertices
	is a tree. Thus, with probability $p_{(b)}=1-o_n(1)$ we can remove a set $E_i^* \subset E_i$ of at most $n^{3/4}+n/\sqrt{n} =n^{3/4}+\sqrt{n} < 2n^{3/4}$ edges from $E_i$ such that
	$(E_i \setminus E_i^*) \cup F$ still constitutes a forest (recall that $F$ is a forest since $\vec{G}_{i-1}$ is successful).
	
	We next consider Item (c). While trivially the maximum out-degree of $E_i \cup F$ is one (being an out-degree one orientation), this is not so for the in-degrees. It could be
	that a vertex whose in-degree in $F$ is $0$ or $1$ has significantly larger in-degree after adding $E_i$. So, we perform the following process for reducing the in-degrees. For each $v \in V(G)$
	whose in-degree in $E_i \cup F$ is $t > 1$, we randomly delete precisely $t-1$ edges of $E_i$
	entering $v$ thereby reducing $v$'s in-degree to $1$ (note: this means that if $v$'s in-degree in $F$ is $1$ we remove all edges of $E_i$ entering it and if $v$'s in-degree in $F$ is $0$ we just keep one edge of $E_i$ entering it, and the kept edge is chosen uniformly at random).
	Let $E_i^{**}$ be the edges removed from $E_i$ by that process. Then we have that
	$(E_i \setminus E_i^{**}) \cup F$ has maximum in-degree at most $1$ and maximum out-degree at most $1$.
	
	We next consider Item (d). For $u \in V(G)$, let $W_u$ denote the $G$-neighbors of $u$ in $X_{i-1}$.
	Since $\vec{G}_{i-1}$ is successful, we have that $|W_u|=(1 \pm o_n(1))rq_{i-1}$.
	Let $Z$ be the number of vertices with in-degree $0$ in $E_i \cup F$. Suppose $u$ has in-degree $0$ in $F$. In order for $u$ to remain with in-degree $0$ in $E_i \cup F$ it must be that each vertex $v \in W_u$ has $\Gamma_i(v) \neq u$.
	The probability of this happening is precisely $(1-1/r)^{|W_u|}=(1-1/r)^{(1 \pm o_n(1))rq_{i-1}}$. Since $\vec{G}_{i-1}$ is successful, there are  $(1 \pm o_n(1))nq_{i-1}$ vertices $u$ with in-degree $0$ in $F$. We obtain that
	$$
	\mathbb{E}[Z] = (1 \pm o_n(1))nq_{i-1}\left(1-\frac{1}{r}\right)^{(1 \pm o_n(1))rq_{i-1}}=(1 \pm o_n(1))nq_i\;.
	$$
	We can prove that $Z$ is tightly concentrated around its expectation as we have done in Lemma \ref{l:neighbors-ok} using martingales. Let $Z_0=\mathbb{E}[Z]$ and let $Z_h$ be the conditional expectation of $Z$ after the edge $(v_h,\Gamma_i(v_h))$ of $E_i$ has been exposed, so that we have
	$Z_{|X_{i-1}|}=Z$. Then, $Z_0,Z_1,\ldots,Z_{|X_{i-1}|}$ is a martingale satisfying the Lipschitz condition (since the exposure of an edge can change the amount of vertices with in-degree $0$ by at most one), so by Azuma's inequality, for every $\lambda > 0$,
	$$
	\Pr \left[\, \left| Z-E[Z] \right| > \lambda |X_{i-1}| \right] < 2e^{-\lambda^2/2}\;.
	$$
	In particular, $Z=(1 \pm o_n(1))nq_i$ with probability $p_{(d)}=1-o_n(1)$ (the $o_n(1)$ term in the probability can even be taken to be exponentially small in $n$).
	
	We next consider Item (e) whose proof is quite similar to the proof of Item (d) above.
	Let $Z_v$ denote the number of
	$G$-neighbors of $v$ with in-degree $0$ in $E_i \cup F$. Since $\vec{G}_{i-1}$ is successful,
	there are $(1 \pm o_n(1))rq_{i-1}$ $G$-neighbors of $v$ with in-degree $0$ in $F$ so the expected value of $Z_v$ is
	$$
	\mathbb{E}[Z_v] = (1 \pm o_n(1))rq_{i-1}\left(1-\frac{1}{r}\right)^{(1 \pm o_n(1))rq_{i-1}}=(1 \pm o_n(1))rq_i\;.
	$$
	As in the previous paragraph, we apply Azuma's inequality to show that $Z_v=(1 \pm o_n(1))rq_i$ with probability $1-o_n(1/n)$, so for all $v \in V(G)$ this holds with probability $p_{(e)}=1-o_n(1)$.
	
	We finally consider Item (f) which is somewhat more delicate as we have to make sure that after removal of $E_i^{**}$, the vertices of $X_{i-1}$ that remain with out-degree $0$ are distributed roughly equally among all neighborhoods of vertices of $G$.
	Fix some $u \in V(G)$, and consider again $W_u$, the $G$-neighbors of $u$ in $X_{i-1}$, recalling that
	$|W_u|=(1 \pm o_n(1))rq_{i-1}$. Suppose $v \in W_u$. We would like to estimate the probability that
	$(v,\Gamma_i(v)) \notin E_i^{**}$. For this to happen, a necessary condition is that $\Gamma_i(v)$
	has in-degree $0$ in $F$. As there are $(1\pm o_n(1))rq_{i-1}$ $G$-neighbors of $v$ with in-degree
	$0$ in $F$, this occurs with probability $q_{i-1}(1\pm o_n(1))$. Now, given that $\Gamma_i(v)$
	has in-degree $0$ in $F$, suppose that $\Gamma_i(v)$ has $t$ in-neighbors in $E_i$. Then, the probability that $(v,\Gamma_i(v)) \notin E_i^{**}$ is  $1/t$. As the probability that $\Gamma_i(v)$ has $t$ in-neighbors in $E_i$ (including $v$) is
	$$
	\binom{s-1}{t-1}\frac{1}{r^{t-1}}\left(1-\frac{1}{r}\right)^{s-t} = (1 \pm o_n(1))\frac{q_{i-1}^{t-1}}{(t-1)!e^{q_{i-1}}}
	$$
	where $s=(1 \pm o_n(1))rq_{i-1}$ is the number of $G$-neighbors of $\Gamma_i(v)$ in $X_{i-1} \setminus \{v\}$. We therefore have that
	$$
	\Pr[(v,\Gamma_i(v)) \notin E_i^{**}] = (1 \pm o_n(1))q_{i-1}\sum_{t=1}^\infty \frac{1}{t} \cdot \frac{q_{i-1}^{t-1}}{(t-1)!e^{q_{i-1}}} =  (1 \pm o_n(1))\left(1-e^{-q_{i-1}}\right)\;.
	$$
	Since $|W_u|=(1 \pm o_n(1))rq_{i-1}$, we have from the last equation that the expected number of neighbors of $u$ with out-degree $0$ in $F \cup (E_i \setminus E_i^{**})$ is
	$$
	(1 \pm o_n(1))rq_{i-1}e^{-q_{i-1}}=(1 \pm o_n(1))rq_i\;.
	$$
	Once again, using Azuma's inequality as in the previous cases, we have that the 
	number of neighbors of $u$ with out-degree $0$ in $F \cup (E_i \setminus E_i^{**})$ is $(1 \pm o_n(1))rq_i$ with probability $1-o(1/n)$, so this holds with probability $p_{(f)}=1-o_n(1)$
	for all $u \in V(G)$.
	
	Finally, we define $F_i = E_i \setminus (E_i^* \cup E_i^{**})$ so items
	(b)-(f) hold for $F_i \cup F$ with probability at least $1-(1-p_{(b)})-(1-p_{(d)})-(1-p_{(e)})-(1-p_{(f)}) > \frac{1}{2}$ (recall that $|E_i^*|=o(n)$ so its
	removal does not change the asymptotic linear quantities stated in items (d),(e),(f)).
\end{proof}

By Lemma \ref{l:a-f}, with probability at least $\frac{1}{2^i}$, we have that $\vec{G}_i$ is successful.
Assuming that $\vec{G}_i$ is successful, let $F_1,\ldots,F_i$ satisfy Definition \ref{def:successful}.
Let $X_i$ be the set of vertices with out-degree $0$ in $\cup_{j=1}^i F_j$. Since $\vec{G}_i$ is successful we have that $|X_i|=(1 \pm o_n(1))nq_i$.
Consider the set of edges of $\vec{G}_{i+1}$ emanating from $X_i$, denoting them by $E_{i+1} = \{(v,\Gamma_{i+1}(v))\,|\, v \in X_i\}$. By adding $E_{i+1}$ to $\cup_{j=1}^i F_j$ we therefore obtain an
out-degree one orientation of $G$, which we denote by $E_{i+1} \cup (\cup_{j=1}^i F_j)$.

\begin{lemma}\label{l:successful-3}
	Let $i \ge 4$ \footnote{We require this assumption so that the value $1-5q_i$ used in the lemma, is positive. Indeed, already $q_4 = 0.162038...$ satisfies this (observe also that $q_i=q_{i-1}/e^{q_{i-1}}$ is monotone decreasing).}.
	Given that $\vec{G}_i$ is successful, and given the corresponding forest $\cup_{j=1}^i F_j$,
	the probability that
	$(E_{i+1} \cup (\cup_{j=1}^i F_j)) \in {\cal H}_{3,nq_i(1 \pm o_n(1))}(G) \cap {\cal H}^*_{\ln n}(G)$
	 is at least
	$$
	(1 - 5q_i-o_n(1))^{nq_i}\;.
	$$
\end{lemma}
\begin{proof}
	Fix an arbitrary ordering of the vertices of $X_i$, say $v_1,\ldots,v_{|X_i|}$.
	We consider the edges $(v_h,\Gamma_{i+1}(v_h))$ one by one, and let
	$E_{i+1,h} \cup (\cup_{j=1}^i F_j)$ be the graph obtained after adding to
	$\cup_{j=1}^i F_j$ the edges $(v_\ell,\Gamma_{i+1}(v_\ell))$ for $1 \le \ell \le h$. Also let $E_{i+1,0} = \emptyset$.
	We say that $E_{i+1,h} \cup (\cup_{j=1}^i F_j)$ is {\em good} if it satisfies the following properties:\\
	(i) The in-degree of each vertex of $E_{i+1,h} \cup (\cup_{j=1}^i F_j)$ is at most $2$.\\
	(ii) Every component of $E_{i+1,h} \cup (\cup_{j=1}^i F_j)$ with fewer than $n/\ln n$ vertices is a tree.\\
	(iii) The number of vertices with in-degree $2$ in $E_{i+1,h} \cup (\cup_{j=1}^i F_j)$ is at most $h$.
	
	Trivially, $E_{i+1,0} \cup (\cup_{j=1}^i F_j)=\cup_{j=1}^i F_j$ is good, since by
	Definition \ref{def:successful}, $\cup_{j=1}^i F_j$ is a forest where the in-degree of each vertex is at most $1$. We estimate the probability that $E_{i+1,h+1} \cup (\cup_{j=1}^i F_j)$ is good given that $E_{i+1,h} \cup (\cup_{j=1}^i F_j)$ is good.
	
	So, consider now the vertex $v_{h+1}$. Since $E_{i+1,h} \cup (\cup_{j=1}^i F_j)$ is assumed good, $v_{h+1}$ has at most $h$ neighbors with in-degree $2$ in $E_{i+1,h} \cup (\cup_{j=1}^i F_j)$.
	Thus, there is a subset $S$ of at least $r-h$ $G$-neighbors of $v_{h+1}$ which still have in-degree at most $1$ in $E_{i+1,h} \cup (\cup_{j=1}^i F_j)$.
	If the component of $v_{h+1}$ in $E_{i+1,h} \cup (\cup_{j=1}^i F_j)$ has fewer than $n/\log n$ vertices, then also remove all vertices of this component from $S$.
	In any case we have that $|S| \ge r-h-n/\ln n$. The probability that $\Gamma_{i+1}(v_{h+1}) \in S$ is therefore at least
	$$
	\frac{r-h-\frac{n}{\ln n}}{r}  = 1 - \frac{h+\frac{n}{\ln n}}{r} \ge 1 - \frac{nq_i(1\pm o_n(1))+\frac{n}{\ln n}}{n/5} = 1 - 5q_i-o_n(1)\;.
	$$
	Now, to have $\Gamma_{i+1}(v_{h+1}) \in S$ means that we are not creating any new components of size smaller than $n/\ln n$, so all components of size at most $n/\ln n$ up until now are still trees and
	furthermore, $E_{i+1,h+1} \cup (\cup_{j=1}^i F_j)$ still has maximum in-degree at most $2$ and at most one additional vertex, namely $\Gamma_{i+1}(v_{h+1})$, which may become now of in-degree $2$, so it has at most $h+1$
	vertices with in-degree $2$. Consequently, $E_{i+1,h+1} \cup (\cup_{j=1}^i F_j)$ is good.
	We have therefore proved that the final $E_{i+1} \cup (\cup_{j=1}^i F_j)$ is good with probability at least
	$$
	(1 - 5q_i-o_n(1))^{|X_i|} \ge (1 - 5q_i-o_n(1))^{nq_i(1 \pm o_n(1))} = (1 - 5q_i-o_n(1))^{nq_i} \;.
	$$
	Finally, notice that the definition of goodness means that
	$(E_{i+1} \cup (\cup_{j=1}^i F_j)) \in {\cal H}_{3,nq_i(1 \pm o_n(1))}(G) \cap {\cal H}^*_{\ln n}(G)$.
\end{proof}

\begin{lemma}\label{l:regular-3}
	Let $K \ge 5$ be given.
	\begin{align*}
		P_{4,0,\ln n}(G)^{1/n} & \ge P_{3,nq_{K-1}(1 \pm o_n(1)),\ln n}(G)^{1/n} \\
		& \ge (1-o_n(1)) \frac{\left(1 - 5q_{K-1}\right)^{q_{K-1}}}{K}\;.
	\end{align*}
\end{lemma}
\begin{proof}
	The first inequality is trivial since an out-degree one orientation with maximum-in degree at most $3$
	has zero vertices with in-degree $4$ or larger. So, we only prove the second inequality.
	Consider the $K$-stage model, and let $i=K-1 \ge 4$.
	By Lemma \ref{l:a-f}, with probability at least $\frac{1}{2^i}$, $\vec{G}_i$ is successful.
	Thus, by Lemma \ref{l:successful-3} and Definition \ref{def:successful}, with probability at least
	$$
	\frac{1}{2^i}(1 - 5q_i -o_n(1))^{nq_i}
	$$
	the following holds:
	There is an out-degree one orientation $\vec{G}$ consisting of edges of $\vec{G_j}$ for $j=1,\ldots,K$,
	at most $nq_i(1 \pm o_n(1))$ of which are edges of $\vec{G}_K$, which is in
	${\cal H}_{3,nq_i(1 \pm o_n(1))}(G) \cap {\cal H}^*_{\ln n}(G)$.
	
	Assuming that this holds, for $j=1,\ldots,K$ let $X_i$ be the set of vertices whose out-edge in
	$\vec{G}$ is from $\vec{G}_j$.
	Now, let $p_1,\ldots,p_K$ with $p_1+\cdots+p_K=1$ be the probabilities associated with the $K$-stage model. The probability that in the second part of the $K$-stage model,
	each vertex $v \in X_j$ will indeed choose $\Gamma_j(v)$ is precisely
	$$
	\prod_{j=1}^K p_j^{|X_j|}\;.
	$$
	Using $p_j=\frac{1}{K}$ for all $p_j$'s and recalling that the final outcome of the $K$-stage model is a completely random element of ${\cal H}(G)$,
	we have that
	$$
	P_{3,nq_i(1 \pm o_n(1)),\ln n}(G)^{1/n} \ge
	\frac{1}{2^i}(1 - 5q_i -o_n(1))^{nq_i}\left(\frac{1}{K}\right)^n\;.
	$$
	Taking the $n$'th root from both sides yields the lemma.
\end{proof}

\begin{proof}[Proof of Theorem \ref{t:regular} for $k \in \{3,4\}$]
	Consider first the case $k=4$ where we will use $K=5$.
	A simple computation gives that $q_{K-1}=q_4 = 0.162038...$ so we have by Lemma \ref{l:regular-3}
	that
	$$
	P_{4,0,\ln n}(G)^{1/n} \ge (1-o_n(1)) \frac{\left(1 - 5q_4\right)^{q_4}}{5}=(1-o_n(1))0.1527...\;.
	$$
	As trivially $P_{4,0,\ln n}(G) \le P_{4,n/28,\ln n}(G)$ we have by Lemma \ref{l:completing}
	that
	$$
	c_4(G)^{1/n} \ge (1-o_n(1))d(G)^{1/n} (1-o_n(1))0.1527... = (1-o_n(1))d \cdot 0.1527...\;.
	$$
	Consider now the case $k=3$ where we will use $K=20$.
	A simple computation gives that $q_{K-1}=q_{19} =0.045821...$ so we have by Lemma \ref{l:regular-3}
	that
	\begin{align*}
	P_{3,n/21,\ln n}(G)^{1/n} & \ge P_{3,nq_{19}(1 \pm o_n(1)),\ln n}(G)^{1/n}\\
	& \ge
	 (1-o_n(1)) \frac{\left(1 - 5q_{19}\right)^{q_{19}}}{20}=(1-o_n(1))0.0494...\;.
	\end{align*}
	We now have by Lemma \ref{l:completing} that
	$$
	c_3(G)^{1/n} \ge (1-o_n(1))d(G)^{1/n} (1-o_n(1))0.0494... = (1-o_n(1))d \cdot 0.0494...\;.
	$$
\end{proof}

It is not too difficult to prove that if we additively increase the minimum degree requirement in Lemma \ref{l:extend-reg} by a small constant, then we can allow for many more vertices of degree $k$ in that lemma. This translates to an increase in the constants $z_3$ and $z_4$.
For example, in the case $k=4$ a minimum degree of $n/5+2$ already increases $z_4$ to about $0.4$
(instead of $z_4=0.1527...$ above) and in the case $k=3$ a minimum degree of $n/4+17$ increases $z_3$ to about $0.2$ (instead of $z_3=0.0494...$ above). However, we prefer to state Theorem \ref{t:regular}
in the cleaner form of minimum degree exactly $n/(k+1)$ for all $k \ge 3$.

\subsection{Regular connected graphs with high minimum degree and $c_k(G)=0$}\label{sub:construct}

In this subsection we show that the requirement on the minimum degree in Theorem \ref{t:regular} is essentially tight. For every $k \ge 2$ and for infinitely many $n$,
there are connected $r$-regular graphs $G$ with $r=\lfloor n/(k+1) \rfloor -2$ for which $c_k(G)=0$.
We mention that a construction for the case $k=2$ is proved in \cite{CS-2013}.

Let $t \ge k+4$ be odd. Let $G_0,\ldots,G_k$ be pairwise vertex-disjoint copies of $K_t$.
Designate three vertices of each $G_i$ for $1 \le i \le k$ where the designated vertices of $G_i$
are $v_{i,0},v_{i,1},v_{i,2}$. Also designate $k+2$ vertices of $G_0$ denoting them by
$v_{0,0},\ldots,v_{0,k+1}$. We now remove a few edges inside the $G_i$'s and add a few edges
between them as follows. For $1 \le i \le k$, remove the edges $v_{i,0}v_{i,1}$ and $v_{i,0}v_{i,2}$
and  remove a perfect matching on the $t-3$ undesignated vertices of $G_i$. Notice that after removal,
each vertex of $G_i$ has degree $t-2$, except $v_{i,0}$ which has degree $t-3$.
Now consider $G_0$ and remove from it all edges of the form $v_{0,0}v_{0,j}$ for $1 \le j \le k+1$.
Also remove a perfect matching on the remaining $t-k-2$ undesignated vertices of $G_0$.
Notice that after removal,
each vertex of $G_0$ has degree $t-2$, except $v_{0,0}$ which has degree $t-k-2$.
Finally, add the edges $v_{0,0}v_{i,0}$ for $1 \le i \le k$. After addition, each vertex has degree 
precisely $t-2$. So the obtained graph $G$ is connected, has $n=(k+1)t$ vertices, and is $r$-regular for $r=n/(k+1)-2$. However, notice that any spanning tree of $G$ must contain all edges
$v_{0,0}v_{i,0}$ for $1 \le i \le k$ and must also contain at least one edge connecting $v_{0,0}$ to
another vertex in $G_0$. Thus, $v_{0,0}$ has degree at least $k+1$ in every spanning tree.

Suppose next that $k \ge 2$ is even and suppose that $t \ge k+5$ be odd.
We slightly modify the construction above. First, we now take $G_0$ to be $K_{t+1}$.
Now, there are $k+3$ designated vertices in $G_0$, denoted by $v_{0,0},\ldots,v_{0,k+2}$.
The removed edges from the $G_i$ for $1 \le i \le k$ stay the same.
The removed edges from $G_0$ are as follows. We remove all edges of the form
$v_{0,0}v_{0,j}$ for $1 \le j \le k+2$. We remove a perfect matching on the vertices
$v_{0,1},\ldots,v_{0,k+2}$. We also remove a Hamilton cycle on the $t-k-2$ undesignated vertices of $G_0$.
Finally, as before, add the edges $v_{0,0}v_{i,0}$ for $1 \le i \le k$. After addition, each vertex has degree 
precisely $t-2$. So the obtained graph $G$ is connected, has $n=(k+1)t+1$ vertices, and is $r$-regular for $r=(n-1)/(k+1)-2=\lfloor n/(k+1) \rfloor-2$.
However, notice that as before, $v_{0,0}$ has degree at least $k+1$ in every spanning tree.

\section{Proofs of Theorems \ref{t:nearly-regular} and \ref{t:non-regular}}\label{sec:nearly-regular}

In this section we prove Theorems \ref{t:nearly-regular} and \ref{t:non-regular}.
As their proofs are essentially identical, we prove them together.
We assume that $k \ge k_0$ where
$k_0$ is a sufficiently large absolute constant satisfying the claimed inequalities.
Although we do not try to optimize $k_0$, it is not difficult to see from the computations that it is a moderate value.

Consider some $\vec{G} \in {\cal H}(G)$. An ordered pair of distinct vertices $u,v \in V(G)$
is a {\em removable edge} if $\Gamma(u)=v$ (so in particular $uv \in E(G)$) and the in-degree of $v$ in $\vec{G}$ is at least $k-1$.

\begin{lemma}\label{l:prob-indeg-2}
	Suppose that $k \ge k_0$. Let $G$ be a graph with minimum degree at least $\delta = n/(k+1)$ and maximum degree at most 
	$\Delta = n(1-3\sqrt{\ln k/k})$. Then with probability at least $\frac{1}{2}$, $\vec{G}$ has
	at most $n/(14k)$ removable edges. The same holds if $G$ has minimum degree at least
	 $\delta^*=\frac{n}{k}(1+3\sqrt{\ln k/k})$ and unrestricted maximum degree.
\end{lemma}
\begin{proof}
	Consider some ordered pair of distinct vertices $u,v \in V(G)$ such that $uv \in E(G)$.
	For that pair to be a removable edge, it must hold
	that: (i) $\Gamma(u)=v$, and (ii) $v$ has at least $k-2$ in-neighbors in $N(v) \setminus u$. As (i)
	and (ii) are independent, and since $\Pr[\Gamma(u)=v] = 1/d(u)$, we need to estimate the number of in-neighbors of $v$ in $N(v) \setminus u$, which is clearly at most $v$'s in-degree in $\vec{G}$.
	So let $D_v$ be the random variable corresponding to $v$'s in-degree in $\vec{G}$.
	Observe that $D_v$ is the sum of independent random variables $D_v=\sum_{w \in N(v)}D_{v,w}$ where $D_{v,w}$ is the indicator variable for the event $\Gamma(w)=v$.
	
	Consider first the case where $G$ has minimum degree at least $\delta$ and maximum degree at most $\Delta$. In particular, $D_v \le X$ where $X \sim \mathrm{Bin}(\Delta,1/\delta)$.
	$$
	\mathbb{E}[X] =  \frac{\Delta}{\delta}=(k+1)(1-3\sqrt{\ln k/k})=k(1-o_k(1))\;.
	$$
	Now let $a=k-2-\mathbb{E}[X] = 3\sqrt{k \ln k}(1-o_k(1))$. Then by Chernoff's inequality (see \cite{AS-2004} Appendix A) we have that for sufficiently large $k$, 
	\begin{align}
	\Pr[D_v \ge k-2] & \le \Pr[X \ge k-2] = \Pr[X-\mathbb{E}[X] \ge a] \nonumber\\
	& \le e^{-a^2/(2\mathbb{E}[X]) + a^3/(2(\mathbb{E}[X])^2)} \nonumber\\
	& \le e^{-(1-o_k(1))9k\ln k/(2k) + (1+o_k(1))27k^{3/2}\ln^{3/2} k/(2k^2)} \nonumber\\
	& \le \frac{1}{k^4} \label{e:1}
	\end{align}
where the last inequality holds for $k \ge k_0$.
	It follows that the probability that $u,v$ is a removable edge is at most
	$(1/d(u))/k^4 \le 1/(\delta k^4) \le 1/(nk^2)$.
	
	Consider next the case where $G$ has minimum degree at least $\delta^*$.
	In particular, $D_v \le X$ where $X \sim \mathrm{Bin}(n,1/\delta^*)$.
	$$
	\mathbb{E}[X] =  \frac{n}{\delta^*}=\frac{k}{1+3\sqrt{\ln k/k}}=k(1-o_k(1))\;.
	$$
	Now let $a=k-2-\mathbb{E}[X] = 3\sqrt{k \ln k}(1-o_k(1))$. So as in \eqref{e:1}, we obtain that $\Pr[D_v \ge k-2] \le 1/k^4$.
	It follows that the probability that $u,v$ is a removable edge is at most $1/(\delta^*k^4) \le 1/(nk^2)$.
	
	As there are fewer than $n^2$ ordered pairs to consider, the expected number of removable edges is
	in both cases is at most $n/k^2$. By Markov's inequality, with probability at least $\frac{1}{2}$, $\vec{G}$ has at most $2n/k^2 \le n/(14k)$ removable edges.
\end{proof}

\begin{lemma}\label{l:red-forest-2}
	Suppose that $k \ge k_0$. Let $G$ be a graph with minimum degree at least $\delta = n/(k+1)$ and maximum degree at most 
	$\Delta = n(1-3\sqrt{\ln k/k})$.
	With probability at least $\frac{1}{4}$, $\vec{G}$ has a spanning forest $F$ such that:\\
	(a) $F$ has maximum in-degree at most $k-2$.\\
	(b) $F$ has at least $n-n/(7k)$ edges.\\
	The same holds if $G$ has minimum degree at least
	$\delta^*=\frac{n}{k}(1+3\sqrt{\ln k/k})$ and unrestricted maximum degree.
\end{lemma}
\begin{proof}
	By Lemma \ref{l:compoents}, with probability at most  $\frac{1}{4}$ we have that $\vec{G}$ has more than $4(k+1)\ln n$ components. Recalling that in $\vec{G}$ each component can be made
	a tree by removing a single edge from its unique directed cycle, with probability at least  $\frac{3}{4}$ we have that $\vec{G}$ can be made acyclic by removing at most  $4(k+1)\ln n$
	edges.
	By Lemma \ref{l:prob-indeg-2}, with probability at least $\frac{1}{2}$, $\vec{G}$ has
	at most $n/(14k)$ removable edges. So, with probability at least $\frac{3}{4}-\frac{1}{2}=\frac{1}{4}$
	we have a forest subgraph of $\vec{G}$ with at least $n-4(k+1)\ln n - n/(14k) \ge n-n/(7k)$ edges in which all removable edges have been removed.
	But observe that after removing the removable edges, each vertex has in-degree at most $k-2$.
\end{proof}

Using the two-stage model, consider the graphs $\vec{G}_1,\vec{G}_2$ as denoted in Section \ref{sec:model}.
For a given $k \ge k_0$, we say that $\vec{G}_1$ is {\em successful}
if it has a spanning forest as guaranteed by Lemma \ref{l:red-forest-2}.
By that lemma, with probability at least $\frac{1}{4}$, we have that $\vec{G}_1$ is successful.
Assuming it is successful, designate a spanning forest $F_1$
of it satisfying the properties of Lemma \ref{l:red-forest-2}. Let $X_1 \subset V(G)$ be the set of vertices with out-degree $0$ in $F_1$. Thus, we have by Lemma \ref{l:red-forest-2} that $|X_1| \le n/(7k)$.
Consider the set of edges of the $\vec{G}_2$ emanating from $X_1$, denoting them by
$E_2 = \{(v,\Gamma_2(v))\,|\, v \in X_1\}$. By adding $E_2$ to $F_1$ we therefore obtain an out-degree one orientation of $G$, which we denote by $E_2 \cup F_1$. The following lemma is analogous to Lemma \ref{l:successful-not3}.

\begin{lemma}\label{l:successful-2}
	Given that $\vec{G}_1$ is successful, and given the corresponding forest $F_1$,
	the probability that $(E_2 \cup F_1) \in {\cal H}_{k,n/(7k)}(G) \cap {\cal H}^*_{\ln n}(G)$ is at least
	$(\frac{5}{6})^{n/(7k)}$.
\end{lemma}
\begin{proof}
	Fix an arbitrary ordering of the vertices of $X_1$, say $v_1,\ldots,v_{|X_1|}$. We consider the edges $(v_i,\Gamma_2(v_i))$ one by one, and let $E_{2,i} \cup F_1$ be the graph obtained after adding to
	$F_1$ the edges $(v_j,\Gamma_2(v_j))$ for $1 \le j \le i$. Also let $E_{2,0} = \emptyset$.
	We say that $E_{2,i} \cup F_1$ is {\em good} if it satisfies the following two properties:\\
	(i) The in-degree of each vertex in $E_{2,i} \cup F_1$ is at most $k-1$.\\
	(ii) Every component of $E_{2,i} \cup F_1$ with fewer than $n/\ln n$ vertices is a tree.\\
	(iii) The number of vertices in $E_{2,i} \cup F_1$ with in-degree $k-1$ is at most $i$.
	
	Note that $E_{2,0} \cup F_1=F_1$ is good, since $F_1$ is a forest where the in-degree of each vertex is at most $k-2$. We estimate the probability that $E_{2,i+1} \cup F_1$ is good given that
	$E_{2,i} \cup F_1$ is good. By our assumption,
	$v_{i+1}$ has at most $i$ neighbors with in-degree $k-1$ in $E_{2,i} \cup F_1$.
	Thus, there is a subset $S$ of at least $d(v_{i+1})-i$ neighbors of $v_{i+1}$ in $G$ which still have in-degree at most $k-2$ in $E_{2,i} \cup F_1$.
	As in Lemma \ref{l:successful-not3}, we may further delete at most $n/\ln n$ vertices from $S$ in case the component of $v_{i+1}$ in $E_{2,i} \cup F_1$ has fewer than $n/\ln n$ vertices
	so that in any case we have that $|S| \ge d(v_{i+1})-i- n/\ln n$. The probability that $\Gamma_2(v_{i+1}) \in S$ is therefore at least
	$$
	\frac{d(v_{i+1})-i-\frac{n}{\ln n}}{d(v_{i+1})} \ge 1 - \frac{i+\frac{n}{\ln n}}{\frac{n}{k+1}}
		\ge 1 - \frac{\frac{n}{7k}+\frac{n}{\ln n}}{\frac{n}{k+1}}
		\ge \frac{5}{6}
	$$
	(note that $d(v_{i+1}) \ge n/(k+1)$ trivially holds also in the assumption of Theorem \ref{t:non-regular}).
	To have $\Gamma_2(v_{i+1}) \in S$ means that we are not creating any new components of size smaller than $n/\ln n$ and that $E_{2,i+1} \cup F_1$ has at most $i+1$ vertices with in-degree
	$k-1$. In other words, it means that $E_{2,i+1} \cup F_1$ is good.
	We have therefore proved that the final $E_2 \cup F_1$ is good with probability at least
	$$
	\left(\frac{5}{6}\right)^{|X_1|} \ge \left(\frac{5}{6}\right)^{n/(7k)}\;.
	$$
	Finally, note that the goodness of $E_2 \cup F_1$ means that it is in ${\cal H}_{k,n/(7k)}(G) \cap {\cal H}^*_{\ln n}(G)$.
\end{proof}

\begin{lemma}\label{l:nearly-regular}
	$$
	P_{k,n/(7k),\ln n}(G)^{1/n} \ge (1-o_n(1))	
	\left(1-\frac{1}{7k}\right)^{1-\frac{1}{7k}}\left(\frac{1}{9k}\right)^{\frac{1}{7k}} = (1-o_n(1))	z^*_k\;.
	$$
\end{lemma}
\begin{proof}
	Considering the two-stage model, we have by Lemma \ref{l:red-forest-2} that with probability at least $\frac{1}{4}$, $\vec{G}_1$  is successful. Thus, by Lemma \ref{l:successful-2}, with probability at least
	$$
	\frac{1}{4}\left(\frac{5}{6}\right)^{n/(7k)}
	$$
	the following holds:
	There is an out-degree one orientation $\vec{G}$ consisting of $x \ge n-n/(7k)$ edges of $\vec{G}_1$
	and hence at most $n/(7k)$ edges of $\vec{G}_2$, which is in ${\cal H}_{k,n/(7k)}(G) \cap {\cal H}^*_{\ln n}(G)$.
	Assuming that this holds, let $X$ be the set of vertices whose out-edge in $\vec{G}$ is from $\vec{G}_1$.
	Now, let $p_1+p_2=1$ be the probabilities associated with the two-stage model where we assume $p_2 < \frac{1}{2}$. The probability that in the second part of the two-stage model,
	each vertex $v \in X$ will indeed choose $\Gamma_1(v)$ and each vertex $v \in V(G) \setminus X$ will indeed choose $\Gamma_2(v)$ is precisely
	$$
	p_1^{x}p_2^{n-x} \ge (1-p_2)^{n-n/(7k)}{p_2}^{n/(7k)}\;.
	$$
	Optimizing, we will choose $p_2=1/(7k)$.
	Recalling that the final outcome of the two-stage model is a completely random element of ${\cal H}(G)$,
	we have that
	$$
	P_{k,n/(7k),\ln n}(G) \ge
	\frac{1}{4}\left(\frac{5}{6}\right)^{\frac{n}{7k}}
	\left(1-\frac{1}{7k}\right)^{n-\frac{n}{7k}}\left(\frac{1}{7k}\right)^{\frac{n}{7k}} \ge
	\frac{1}{4}
	\left(1-\frac{1}{7k}\right)^{n-\frac{n}{7k}}\left(\frac{1}{9k}\right)^{\frac{n}{7k}} \;.
	$$
	Taking the $n$'th root from both sides yields the lemma.
\end{proof}

\begin{proof}[Proof of Theorems \ref{t:nearly-regular} and \ref{t:non-regular}]
	Combining Lemma \ref{l:nearly-regular} and Lemma \ref{l:completing} we have that
	$$
	c_k(G)^{1/n} \ge (1-o_n(1))d(G)^{1/n}z^*_k\;.
	$$
\end{proof}

\section*{Acknowledgment}
The author thanks the referees for their useful comments.

\end{document}